\documentclass[12pt,leqno]{amsart}
\usepackage{amssymb,amsthm,amsmath,latexsym}
\newtheorem{theorem}{\sc Theorem}[section]

\newtheorem{lemma}[theorem]{\sc Lemma}

\newtheorem{ex}[theorem]{\sc Example}

 \title[Exponent of Self-similar finite $p$-groups]{Exponent of Self-similar finite $p$-groups}

\author{Alex Carrazedo Dantas}
\address{Universidade de Bras\'ilia, Departamento de Matem\'atica, Bras\'ilia - DF,  70910-900, Brazil}
\email{alexcdan@gmail.com}

\author{Emerson de Melo} 
\address{Universidade de Bras\'ilia, Departamento de Matem\'atica, Bras\'ilia - DF, 70910-900, Brazil}
\email{emerson@mat.unb.br}
\subjclass[2010]{20D05; 20D45.}
\keywords{powerful groups, self-similar}

\begin{document}
\begin{abstract}
Let $p$ be a prime and $G$ a pro-$p$ group of finite rank that admits a faithful, self-similar action on the $p$-ary rooted tree. We prove that if the set $\{g\in G \ | \ g^{p^n}=1\}$ is a nontrivial subgroup for some $n$, then $G$ is a finite $p$-group with exponent at most $p^n$. This applies in particular to power abelian $p$-groups.

\end{abstract}
\maketitle
\section{Introduction}

A group $G$ is self-similar of some integer degree $m$ if it has a faithful representation on an infinite regular one-rooted $m$-tree $\mathcal{T}_{m}$ such that the representation is state-closed and is transitive on the tree's first level. Equivalently, a group $G$ is self-similar of degree $m$ if there exists a subgroup $H$ of $G$ of index $m$ and a homomorphism $f$ from $H$ to $G$ such that the only subgroup of $H$, normal in $G$ and $f$-invariant is the trivial subgroup.

Faithful self-similar representations are known for many individual finitely generated groups ranging from the torsion groups of Grigorchuk and Gupta-Sidki to free groups. Such representations have also been  studied for the family of abelian groups, of finitely generated nilpotent groups, wreath product of abelian groups, as well as for arithmetic groups, see \cite{DS}, \cite{DS1}, \cite{BS}, \cite{BeS}, and \cite{K} for more references.

Let $p$ be a prime and $G$ a pro-$p$ group. Consider the following subgroups of $G$: $\Omega_n(G)=\langle g\in G \ | \ g^{p^n}=1 \rangle$ and $G^{p^{n}}=\langle g^{p^{n}} \ | \ g\in G \rangle$. If $G$ is abelian, we have the following simple description of these groups:
\begin{enumerate}
    \item[(1)]  $\Omega_n(G)=\{ g\in G \ | \ g^{p^n}=1 \}$.
    \item[(2)]  $G^{p^{n}}=\{ g^{p^{n}} \ | \ g\in G \}$.
\end{enumerate}
Moreover, we have that:
\begin{enumerate}
     \item[(3)]  $|G^{p^{n}}|=|G:\Omega_n(G)|$ for all $i$.
\end{enumerate}
We call a finite $p$-group $G$ power abelian if it satisfies these three conditions for all $i$. It is a very interesting problem to determine which groups are power abelian. In \cite{PH} P. Hall introduced the notion of regular groups and showed that they are power abelian. Recently, it was proved in \cite{JJ} that potent $p$-groups are also power abelian. Recall that a finite $p$-group $G$ is potent if $[G,G]\leq G^4$ for $p=2$ or $\gamma_{p-1}(G)\leq G^p$ for $p>2$. Note that the family of potent $p$-groups includes powerful $p$-group ($[G,G]\leq G^4$ for $p=2$ or $[G,G]\leq G^p$ for $p>2$). More information on finite $p$-groups satisfying property $(1)$, $(2)$ or $(3)$ can be found in \cite{AM}. 

Not much is known about self-similarity of finite $p$-groups. We can mention the work of Sunic \cite{S}, he proved that a finite $p$-group $G$ is self-similar of degree $p$ with an abelian first level stabilizer if and only $G$ is a split extension of an elementary abelian group by a cyclic group of order $p$. Actually, in \cite[Theorem 2.2]{cpa} it was  showed that this characterization extends to all finite p-groups with an abelian maximal subgroup, without needing to specify any properties of its action on the tree. In \cite{cpa} it was also proved that if $G$ is a self-similar finite $p$-group of rank $r$ (self-similar of degree $p$), then its order is bounded by a function of $p$ and $r$. In other words, they proved that, for every prime $p$, there are only finitely many self-similar $p$-groups of a given rank.

In the present article we address the case where $G$ is a self-similar pro-$p$ group of degree $p$. We show that the exponent of $G$ is determined by its ``power structure''. Recall that a pro-$p$ group has finite rank if there exists a positive integer $r$ such that any closed subgroup has a topological generating set with no more than $r$ elements.  

\begin{theorem}\label{th1}
Let $p$ be a prime and $G$ a pro-$p$ group of finite rank such that $\{g\in G \ | \ g^{p^n}=1\}$ is a nontrivial subgroup of $G$ for some $n$. If $G$ is self-similar  of degree $p$, then $G$ is a finite $p$-group and has exponent at most $p^n$.
\end{theorem}

If $G$ is a finite $p$-group, then of course $\Omega_n(G)=\{g\in G \ | \ g^{p^n}=1\}$ for some $n$. Using Theorem \ref{th1} we obtain that the exponent of a self-similar finite $p$-group satisfying property $(1)$ is equal to $p$. In particular, the exponent of a finite self-similar power abelian $p$-group is equal to $p$. In some sense Theorem \ref{th1} generalize the following results: 

\begin{itemize}
    \item A self-similar torsion abelian group of degree $p$ has exponent $p$, as proved by  Brunner and Sidki \cite{BS}.
    \item If a finitely generated nilpotent group $G$ is self-similar of degree $p$, then $G$ is either free abelian or a finite $p$-group, as proved by Berlatto and Sidki \cite[Corollary 2]{BeS}.
\end{itemize}

It should be mentioned that if $G$ is self-similar of degree $p$ and the set $\{g\in G \ | \ g^{p^n}=1\}$ is not a subgroup, then $G$ need not be a torsion free group. For example the pro-2 group $\mathbb{Z}_2 \rtimes C_2$ (dihedral pro-2 group) is self-similar of degree 2 and the set of elements of order 2 is not a subgroup.

We use Theorem \ref{th1} to obtain a generalization of  the above mentioned result of \cite{S}.

\begin{theorem}\label{th2}
Let $p$ be a prime. If  $G$ is a self-similar finite $p$-group of degree $p$ such that the first level stabilizer is power abelian, then $G$ is a split extension of a $p$-group of exponent $p$ by a cyclic group of order $p$. 
\end{theorem} 

In general, $p$-groups with a maximal power abelian subgroup are not self-similar. In particular, we have the following example. Let $G$ be an extra-special $p$-group of exponent $p$ and order at least $p^5$. Let $H$ be a maximal subgroup of $G$. In this case we have that $[G,G]=[H,H]$. Then the subgroup $[H,H]$ is normal and $f$-invariant for any homomorphism $f$ from $H$ to $G$ and $G$ is not self-similar. On the other hand, in Section 4 we provide examples of self-similar finite $p$-groups with nonabelian but power abelian first level stabilizer.

\section{Preliminaries}

\noindent {\bf Self-similar groups:} Let $\mathcal{A}_{m} = Aut(\mathcal{T}_{m})$ be the group of automorphisms  of the infinite regular one-rooted $m$-ary tree $\mathcal{T}_{m}$. The group  $\mathcal{A}_{m}$ is isomorphic to  $\mathcal{A}_{m} \wr S_{m}$, where $S_{m}$ is the symmetric group of degree $m$. Then an element $\alpha$ of $\mathcal{A}_{m}$ is given by $\alpha = (\alpha_{0}, \alpha_{1}, ..., \alpha_{m-1}) \sigma(\alpha)$, where $\alpha_{i} \in \mathcal{A}_{m}$ and $\sigma(\alpha)$ is the action of $\alpha$ in the first level of $\mathcal{T}_{m}$. A subgroup $G$ of $\mathcal{A}_{m}$ is state-closed if for all $\alpha = (\alpha_{0}, \alpha_{1}, ..., \alpha_{m-1}) \sigma(\alpha) \in G$ implies $\alpha_{i} \in G$. Also, let $P(G)$ denote the permutation group induced by $G$ on the first level of the tree. We say $G$ is transitive provided $P(G)$ is transitive. A subgroup $G$ of $\mathcal{A}_{m}$ is said to be self-similar provided it is a state-closed and transitive group. \\

\noindent {\bf Virtual endomorphism:} Let $G$ be a group and let $H$ be a subgroup of $G$ of index $m$. A homomorphism $f: H \rightarrow G$ is called a virtual endomorphism of $G$. Let $T = \{t_{0}, t_{1}, ..., t_{m-1}\}$ be a transversal of $H$ in $G$. Thus 
$$\varphi: g \mapsto g^{\varphi} = \left([t_{i} g t_{(i)\sigma(g)}^{-1}]^{f \varphi} \right)_{0 \leq i \leq m-1} \sigma(g), \, \forall g \in G,$$
is a homomorphism of $G$ in $\mathcal{A}_{m}$, where $\sigma(\alpha)$ is the permutation of $T$ induced by $g$. The kernel of $\varphi$ is the subgroup $f-core(H)$ given by
$$f-core(H) = \langle K \leq H \mid H \vartriangleleft G, K^{f} \leq K \rangle.$$
If $f-core(H) = \{1\}$, then $f$ is called simple endomorphism. Note that $H^{\varphi}  = Fix_{G^{\varphi}}(0) = \{g^{\varphi} \in G^{\varphi} \mid (0)\sigma(g) = 0\}$. The next theorem is well-known (see \cite{Ne0} and \cite{Ne} for further details).
\begin{theorem}
A group $G$ is self-similar of degree $m$ if and only if there exists a simple endomorphism $f: H \rightarrow G$, where $H$ is a subgroup of $G$ of index $m$.
\end{theorem}

\begin{ex} \label{ex 2.3}
Consider the finite p-group $G$ with the presentation
$$\langle a, b \mid [a, b] = c, [a, c] = [b, c] = a^{p} = b^{p} = c^{p} = 1\rangle,$$
$H$ its subgroup $\langle a, c \rangle$ of index $p$ and define $f: H \rightarrow G$ the homomorphism that extend the map $a \mapsto c$, $c \mapsto b$. Note that $f$ is a simple virtual endomorphism. Hence $G$ is a self-similar group of degree $p$. If $T = \{1, b, ..., b^{p-1}\}$, then $H$, $f$ and $T$ iduce a representation $\varphi: G \rightarrow \mathcal{A}_{p}$ such that
$$G^{\varphi} = \langle \alpha = (\gamma, \gamma \beta^{p-1}, \gamma \beta^{p-2}, ..., \gamma \beta), \gamma = (\beta, ..., \beta), \beta = (0 \, 1 \, ... \, p-1) \rangle.$$
\end{ex}

The following two lemmas was proved in \cite[Lemma 3]{S} and \cite[Proposition 2]{BeS}, respectively. For the reader's convenience we supply a proof.

\begin{lemma}\label{split}
Let $G$ be a pro-$p$ group and $f:H\rightarrow G$ a simple virtual endomorphism of degree $p$. If $H$  contains elements of order $p$, then $H^{f}\setminus H$ contains elements of order $p$. In particular, $G=H \rtimes \langle a \rangle$ for some element $a$ of order $p$.
\end{lemma}
\begin{proof}
Let $P$ be the set of elements in $H$ of order dividing $p$. The set $P$ contains nontrivial elements and the subgroup $\langle P \rangle$ is a nontrivial, characteristic subgroup of
$H$. Therefore, $\Omega_1(P)$ is nontrivial, normal subgroup of $G$. The elements of $P$ are
mapped under $f$ to elements of order dividing $p$. If all elements of $P$ are mapped inside $H$, then they are mapped to other elements in $P$. In that case $\Omega_1(P)$ would be a nontrivial, normal, $f$-invariant subgroup of $G$, a  contradiction. Therefore there exists an element in $P$ that is mapped under $f$ outside of $H$. The image of such an element has order $p$ and this completes the proof.
\end{proof}

\begin{lemma}\label{virt2}
Let $G$ be a pro-$p$ group and $f:H\rightarrow G$ a simple virtual endomorphism of degree $p$ of $G$. Then $f:H\cap H^{f} \rightarrow H^{f}$ is also a simple virtual endomorphism of degree $p$.
\end{lemma}
\begin{proof}
Let $a$ be an element of order $p$ in $H^{f} \setminus H$. Such an element exists by Lemma \ref{split}. Since $[G:H]=p$ and $H^{f}\neq H^{f}\cap H$ we have that $[H^{f}: H^{f}\cap H]=p$. Let $K$ be a subgroup of $H\cap H^{f}$, normal in $H^{f}$, with $K^{f}\leq K$. Thus, $K\leq K^{f^{-1}}$ and $K^{f^{-1}}$ is normal in $H$. Then $$K=K^a\leq (K^{f^{-1}})^a=(K^{f^{-1}})^{ha}$$ for all $h\in H$, that is $K\leq (K^{f^{-1}})^{g}$ for all $g\in G$. Let $M=\cap_{g\in G}(K^{f^{-1}})^{g}$. Then $K \leq M \leq K^{f^{-1}}$. Thus $M$ is a normal subgroup of $G$ such that $M^{f}\leq K \leq M$. Therefore, $M=1=K$.
\end{proof}

\begin{theorem}\label{exp}
Let $G$ be a pro-$p$ group and $f:H\rightarrow G$ a simple virtual endomorphism of degree $p$ of $G$. Then $H^{p^n}=1$ if and only if $(H^{f})^{p^n}=1$.
\end{theorem}
\begin{proof}
It is clear that if $H^{p^n}=1$, then $(H^{f})^{p^n}=1$. Suppose that $(H^{f})^{p^n}=1$ e let $K$ be the kernel of $f$. Thus $H^{p^n} \leq K$ since $(H^{p^n})^{f}=1$. On the other hand, $H^{p^n}$ is a characteristic subgroup of $H$ and then it is a normal subgroup of $G$. Therefore, $H^{p^n}=1$   
\end{proof}


\section{Proof of Theorem \ref{th1} and Theorem \ref{th2}}

Assume the hypothesis of Theorem \ref{th1}. Thus $G$ is a pro-$p$ group of finite rank such that the set $\{g\in G \ | \ g^{p^n}=1\}$ is a nontrivial subgroup for some $n$ and there exists a simple virtual endomorphism $f:H\rightarrow G$ of degree $p$. We wish to show that $G$ has exponent $p^n$. Clearly in this case we have $\Omega_n(G)=\{g\in G \ | \ g^{p^n}=1\}$.

A pro-$p$ group of finite rank has a characteristic torsion free subgroup of finite index, see for example \cite[Corollary 4.3]{DSMS}. Then $G$ contains a torsion free subgroup $F$ of finite index. In particular, we have that $\Omega_n(G)\cap F=1$ and then $\Omega_n(G)$ is finite.

We argue by induction on the order of $\Omega_n(G)$. Note that $\Omega_n(H)=H\cap \Omega_n(G)$. First assume that $\Omega(H)= 1$. In this case we have $G=H \times\langle a \rangle$ where $a$ is an element of order $p$ since $\Omega_n(G)$ is normal in $G$ and $H$ has index $p$. We obtain that $G^p=H^p$. On the other hand we have that $(H^p)^f= (H^f)^p\leq G^p$. Then $H^p$ is an $f$-invariant normal subgroup of $G$. Therefore $H^p=1$ and we conclude that $G=\Omega_1(G)$.

Now, assume that $\Omega(H)\neq 1$. In this case we have that $G=H \rtimes\langle a \rangle$ where $a$ is an element of order $p$ by Lemma \ref{split}. If $\Omega_n(H)^f=\Omega_n(H)$, then $\Omega_n(H)=1$. Thus $|\Omega_n(H\cap H^f)|< |\Omega_n(H)|$. By induction and using Lemma \ref{virt2} we obtain that $H^f$ has exponent at most $p^n$. Now, by Lemma \ref{exp} $H$ has exponent at most $p^n$. Therefore $G=\{g\in G \ | \ g^{p^n}=1\}$ and this completes the proof of Theorem \ref{th1}.

Now, assume the hypothesis of Theorem \ref{th2}. Thus $G$ is a finite $p$-group, $H$ is power abelian and $f:H\rightarrow G$ a simple virtual endomorphism of degree $p$. Thus $\Omega_1(H)=\{h\in H \ | \ h^{p}=1\}$ and then $\Omega_1(H^{f})=\{g\in H^{f} \ | \ g^{p}=1\}$. Now, using Lemma \ref{virt2} and Theorem \ref{th1} we conclude that $H^{f}$ has exponent $p$. Therefore $H$ has exponent $p$ by Lemma \ref{exp} and the proof is complete. \\

\section{Discussion}

By \cite{S} a extra-special $p$-group of order $p^3$ and exponent $p$ and the wreath product $C_{p} \wr C_{p}$ are self-similar $p$-groups of degree $p$ with abelian first level stabilizer, since they have a maximal elementary abelian subgroup. It is clear that  they have exponent $p$ and $p^{2}$ respectively. In \cite[Proposition 2.9.3]{Ne} it was proved that the direct power of a self-similar group is again self-similar. In particular, it is easy to see that a direct power of extra-special $p$-groups of order $p^3$ and exponent $p$ is self-similar with a non-abelian but power abelian first level stabilizer. The next result shows that it is also possible to construct self-similar $p$-groups of exponent $p^2$ with a non-abelian but power abelian first level stabilizer.

\begin{theorem}
If $G$ is a self-similar finite $p$-group of degree $p$, then $P = G \wr C_{p}$ is also.
\end{theorem}

\begin{proof}
By Theorem 2.2 and Lemma 2.4, there exist a subgroup $H$ of index $p$ in $G$, a simple endomorphism $f: H \rightarrow G$ and an element $a$ in $G$ of order $p$ such that $G = H \rtimes \langle a \rangle$. Let $T = \{1, a, a^{2}, ..., a^{p - 1}\}$ a transverse of $H$ in $G$. Then $H$, $f$ and $T$ induce the representation $\varphi: G \rightarrow \mathcal{A}_{p}$ that extend the following map
$$a \mapsto a^{\varphi} = (0 \, \, 1 \, \, ... \, \, p-1) = \sigma$$
$$h \mapsto h^{\varphi} = (h^{f \varphi}, h^{a^{p-1}f \varphi}, ..., h^{af \varphi}) = (k_{0}, k_{1}, ..., k_{p-1}),$$
where $h \in H$. Consider the homomorphism
$$\theta : G^{\varphi} \rightarrow \mathcal{A}_{p}$$
$$\, \, \, \, \, \, \, \, \, \, \, \, \, \, \, \, \, \, \, \, \, \, \, \, \, \, \, \, \, \, \, \, g^{\varphi} \mapsto (g^{\varphi}, e, ..., e).$$
We affirm the group $P = \langle G^{\varphi \theta}, \sigma \rangle$ is state-closed. In fact, the states of $P$ are the elements of $G^{\varphi}$. Since that $a^{\varphi} = \sigma \in P$, we need only show that $h^{\varphi} \in P$, for all $h \in H$. Note that
$$h^{\varphi} = (k_{0}, k_{1}, ..., k_{p-1}) = k_{0}^{\theta} k_{1}^{\theta \sigma} ... k_{p-1}^{\theta \sigma^{p-1}} \in \langle G^{\varphi \theta}, \sigma \rangle.$$
The assertion follows. Since $P$ is transitive and isomorphic to $G \wr C_{p}$, the results follows. 
\end{proof}

\end{document}